\documentclass[reqno]{amsart}%
\usepackage{amssymb,mathtools,lineno}

\usepackage[colorlinks=true,citecolor=blue,linkcolor=green,urlcolor=red]{hyperref}

\allowdisplaybreaks[4]

\theoremstyle{plain}
\newtheorem{thm}{Theorem}[section]
\newtheorem{conj}{Conjecture}[section]
\newtheorem{lem}{Lemma}[section]

\theoremstyle{remark}
\newtheorem{rem}{Remark}[section]

\DeclareMathOperator{\td}{d\!}
\DeclareMathOperator{\te}{e}

\begin{document}

\title[Monotonicity of sectional curvature of Fisher--Rao manifold]
{Convexity of a ratio function involving polygamma functions and monotonicity of the sectional curvature of the Fisher--Rao manifold}

\author[F. Qi]{Feng Qi}
\address{School of Mathematics and Physics, Hulunbuir University, Hulunbuir 021008, Inner Mongolia, China;
17709 Sabal Court, University Village, Dallas, TX 75252-8024, USA}
\email{\href{mailto:qifeng618@gmail.com}{qifeng618@gmail.com}}
\urladdr{\url{https://orcid.org/0000-0001-6239-2968}}

\begin{abstract}
This paper proves that the sectional curvature of the Fisher--Rao manifold of beta distributions, restricted to the diagonal, is strictly decreasing. The proof is based on a general convexity lemma for ratios involving positive decreasing functions and is applied to a quotient of polygamma functions.
\end{abstract}

\keywords{sectional curvature, monotonicity, lower bound, Fisher--Rao manifold, beta distribution, polygamma function}

\subjclass{Primary 53C25; Secondary 26A48, 26A51, 33B15, 60E05}

\thanks{This paper was typeset using \AmS-\LaTeX}

\maketitle

\section{Introduction}
The Fisher--Rao manifold of beta distributions is the parameter space of the beta family endowed with the Fisher--Rao metric. This geometric structure arises in information geometry, a growing field that examines statistical models through the lens of differential geometry. This manifold provides insight into the intrinsic geometry of beta distributions and facilitates further analysis of their statistical and inferential properties.
\par
Let $M=\{(x,y): x,y>0\}$ denote the first quadrant of $\mathbb{R}^2$. Equip $M$ with the Fisher metric
\begin{equation*}
\td s^2=\psi'(x)\td x^2+\psi'(y)\td y^2-\psi'(x+y)(\td x+\td y)^2,
\end{equation*}
where, for $\Re(z)>0$,
\begin{equation*}
\psi(z)=[\ln\Gamma(z)]'=\frac{\Gamma'(z)}{\Gamma(z)}\quad\text{and}\quad
\Gamma(z)=\int_0^{\infty}t^{z-1}\te^{-t}\td t.
\end{equation*}
\par
In~\cite[Proposition~3]{brigant2019fisherrao}, \cite[Proposition~14]{brigant2020fisherraoV1}, \cite[Proposition~15]{brigant2020fisherrao}, and~\cite[Proposition~15]{brigant2021fisherrao}, the authors computed the sectional curvature $K(x,y)$ of the Fisher metric $\td s^2$ on the beta manifold $\bigl(M,\td s^2\bigr)$ and obtained
\begin{equation}\label{curvature-polygamma-exp}
K(x,y)=\frac{1}{4}\frac{\psi''(x)\psi''(y)\psi''(x+y) \Bigl[\frac{\psi'(x)}{\psi''(x)}+\frac{\psi'(y)}{\psi''(y)}-\frac{\psi'(x+y)}{\psi''(x+y)}\Bigr]}
{[\psi'(x)\psi'(x+y)+\psi'(y)\psi'(x+y)-\psi'(x)\psi'(y)]^2}, \quad (x,y)\in M.
\end{equation}
\par
In~\cite[Proposition~5]{brigant2019fisherrao}, \cite[Theorem~6]{brigant2020fisherraoV1}, \cite[Theorem~6]{brigant2020fisherrao}, and~\cite[Theorem~6]{brigant2021fisherrao}, the authors established that the sectional curvature $K(x,y)$ is negative and bounded from below, and proposed the following conjecture.

\begin{conj}[{\cite[Proposition~5]{brigant2019fisherrao}, \cite[pp.~12--13]{brigant2020fisherraoV1}, \cite[pp.~13--14]{brigant2020fisherrao}, and~\cite[p.~14]{brigant2021fisherrao}}]\label{Alice-lower-bound-conj}
For $x,y>0$, the sectional curvature $K(x,y)$ given in~\eqref{curvature-polygamma-exp} satisfies the following properties:
\begin{enumerate}
\item
it has the lower bound $-\frac12$; more precisely, $K(x,y)>-\frac12$;
\item
it is decreasing in each of $x$ and $y$.
\end{enumerate}
\end{conj}

In the review article~\cite{manifold-gamma.tex}, together with about ten papers cited therein, several functions involving the polygamma functions $\psi^{(n)}(x)$---originating from the study of the sectional curvature $K(x,y)$---were investigated extensively.
\par
Taking $y=x$ in~\eqref{curvature-polygamma-exp} yields
\begin{equation}\label{K(x-x)-Eq}
\mathcal{K}(x)=K(x,x)=\frac{1}{4}\frac{\psi''(x)}{[\psi'(x)]^2} \frac{2\psi'(x)\psi''(2x)-\psi'(2x)\psi''(x)}{[\psi'(x)-2\psi'(2x)]^2}, \quad x\in(0,\infty).
\end{equation}
In~\cite[Theorem~3.2]{RIMA-D-20-01130.tex}, the second factor $\frac{\psi''(x)}{[\psi'(x)]^2}$ in~\eqref{K(x-x)-Eq} was proved to be decreasing and to map $(0,\infty)$ onto $(-1,0)$.
In~\cite[Theorem~4.1]{RIMA-D-20-01130.tex}, the sharp double inequality
\begin{equation}\label{sectional-doub-ineq}
0<\frac{2\psi'(x)\psi''(2x)-\psi'(2x)\psi''(x)}{[\psi'(x)-2\psi'(2x)]^2}<2, \quad x\in(0,\infty)
\end{equation}
was established. Consequently, the double inequality
\begin{equation}\label{sect-curvat-ineq}
0>\mathcal{K}(x)>-\frac{1}{2}, \quad x\in(0,\infty)
\end{equation}
was derived in~\cite[Theorem~5.1]{RIMA-D-20-01130.tex}. This may be viewed as the initial and partial confirmation of Conjecture~\ref{Alice-lower-bound-conj}.
\par
In this paper, we aim to prove that the sectional curvature $\mathcal{K}(x)$ defined in~\eqref{K(x-x)-Eq} is decreasing for $x\in(0,\infty)$, a stronger conclusion than the double inequality~\eqref{sect-curvat-ineq}. This result can be regarded as the second, though partial, confirmation of Conjecture~\ref{Alice-lower-bound-conj}.

\section{Two lemmas}
To achieve this goal, we need the following two lemmas.

\begin{lem}\label{lem:general}
Let $f(x)$ be three times continuously differentiable and $H(x)=-\frac{f(x)}{f'(x)}$ on $(0,\infty)$. If
$$
f(x)>0,\quad f'(x)<0, \quad 0<H'(x)<1, \quad H''(x)\geq 0
$$
on $(0,\infty)$, then, for every given $a>0$, the function
$$
F_a(x)=\frac{1}{1-f(x+a)/f(x)}
$$
is strictly convex on $(0,\infty)$.
\end{lem}

\begin{proof}
Define $\Lambda(x)=\ln\frac{f(x)}{f(x+a)}$ for $x\in(0,\infty)$.
Since $\frac{f'(x)}{f(x)}=-\frac1{H(x)}$, we have
$$
\Lambda(x)=\ln f(x)-\ln f(x+a)
=\int_{x+a}^{x}\frac{f'(t)}{f(t)}\td t
=
\int_x^{x+a}\frac{\td t}{H(t)}>0.
$$
Consequently, it follows that
$$
\Lambda'(x)=\frac1{H(x+a)}-\frac1{H(x)}<0
$$
and
$$
\Lambda''(x)
=\frac{H'(x)}{H^2(x)}-\frac{H'(x+a)}{H^2(x+a)}.
$$
Therefore,
\begin{equation}\label{eq:lambda-bound}
\begin{aligned}
\frac{\Lambda''(x)}{[\Lambda'(x)]^2}
&=\frac{H'(x)H^2(x+a)-H'(x+a)H^2(x)}{[H(x+a)-H(x)]^2}\\
&\leq \frac{H'(x)[H^2(x+a)-H^2(x)]}{[H(x+a)-H(x)]^2}\\
&=H'(x)\frac{H(x+a)+H(x)}{H(x+a)-H(x)}\\
&=H'(x)\coth\frac{L(x)}{2}, \quad x\in(0,\infty),
\end{aligned}
\end{equation}
where
$$
L(x)=\ln\frac{H(x+a)}{H(x)}>0, \quad x\in(0,\infty).
$$
Moreover, we have
$$
L(x)=\int_x^{x+a}\frac{H'(t)}{H(t)}\td t
\geq H'(x)\int_x^{x+a}\frac{\td t}{H(t)}
=H'(x)\Lambda(x).
$$
Hence, we arrive at
\begin{equation}\label{LL-ineq}
\frac{\Lambda(x)}{2}\leq\frac{L(x)}{2H'(x)}, \quad x\in(0,\infty).
\end{equation}
\par
Let
$$
G(z)=z\coth z>0, \quad z\in(0,\infty).
$$
Direct differentiation gives
$$
G'(z)=\frac{\sinh z\cosh z-z}{\sinh^2z}
$$
and
$$
(\sinh z\cosh z-z)'=2\sinh^2z>0.
$$
This implies that $G'(z)>0$, and hence $G(z)$ is strictly increasing for $z\in(0,\infty)$.

Since $0<H'(x)<1$, it follows that
$$
\frac{L(x)}{2H'(x)}>\frac{L(x)}{2}.
$$
The strict increase of $G(z)$ results in
$$
\coth\frac{L(x)}{2H'(x)}>H'(x)\coth\frac{L(x)}{2}.
$$
Since $\coth z$ is strictly decreasing on $(0,\infty)$, applying the inequality~\eqref{LL-ineq} yields
$$
\coth\frac{\Lambda(x)}{2}
\geq
\coth\frac{L(x)}{2H'(x)}
>
H'(x)\coth\frac{L(x)}{2}.
$$
Combining this with~\eqref{eq:lambda-bound} leads to the inequality
\begin{equation}\label{eq:key-inequality}
\frac{\Lambda''(x)}{[\Lambda'(x)]^2}<\coth\frac{\Lambda(x)}{2}, \quad x\in(0,\infty).
\end{equation}
\par
Now let
$$
q(x)=\te^{-\Lambda(x)}=\frac{f(x+a)}{f(x)}, \quad x\in(0,\infty).
$$
Then $0<q(x)<1$ and
$$
F_a(x)=\frac1{1-q(x)}.
$$
Furthermore, we have
$$
q'(x)=-q(x)\Lambda'(x)\quad\text{and}\quad
q''(x)=q(x)\bigl([\Lambda'(x)]^2-\Lambda''(x)\bigr).
$$
Thus, the second derivative
$$
F_a''(x)=\frac{q''(x)}{[1-q(x)]^2}+\frac{2[q'(x)]^2}{[1-q(x)]^3}
=
\frac{q(x)}{[1-q(x)]^2}
\left[
\frac{1+q(x)}{1-q(x)}[\Lambda'(x)]^2-\Lambda''(x)
\right].
$$
Since
$$
\frac{1+q(x)}{1-q(x)}
=
\coth\frac{\Lambda(x)}{2},
$$
we obtain
$$
F_a''(x)
=
\frac{\te^{-\Lambda(x)}[\Lambda'(x)]^2}{[1-\te^{-\Lambda(x)}]^2}
\biggl[\coth\frac{\Lambda(x)}{2}-\frac{\Lambda''(x)}{[\Lambda'(x)]^2}\biggr].
$$
The bracket is strictly positive by~\eqref{eq:key-inequality}.
Therefore, it follows that $F_a''(x)>0$. The proof of Lemma~\ref{lem:general} is thus complete.
\end{proof}

\begin{rem}
The upper bound $1$ in the condition $0<H'(x)<1$ in Lemma~\ref{lem:general} is necessary. For example, if we take $f(x)=\frac{1}{\sqrt{x}\,}$, then $H(x)=2x$,
\begin{equation*}
F_a(x)=\frac{x+a+\sqrt{x(x+a)}}{a}, \quad F_a''(x)=-\frac{a}{4[x(x+a)]^{3/2}}<0
\end{equation*}
for $a>0$. The conclusion of Lemma~\ref{lem:general} fails.
\par
If we take $f(x)=\frac{1}{x}$, then $H(x)=x$, $H'(x)=1$, and $F_a(x)=\frac{x+a}{a}$ whose second derivative is just $0$. This shows that the strict inequality $H'(x)<1$ in Lemma~\ref{lem:general} cannot be dropped in order to guarantee the strict convexity of $F_a$.
\end{rem}

Applying Lemma~\ref{lem:general} with $f(x)=\bigl|\psi^{(n)}(x)\bigr|$, we obtain the following lemma.

\begin{lem}\label{OpenProb-convex}
For fixed $\alpha>0$ and $n\in\mathbb{N}=\{1,2,\dotsc\}$, the function
\begin{equation}\label{Phi-n-alpha(x)}
\Phi_{n,\alpha}(x)=\frac{1}{1-{\psi^{(n)}(x+\alpha)}/{\psi^{(n)}(x)}}
\end{equation}
is strictly increasing and strictly convex on $(0,\infty)$.
\end{lem}

\begin{proof}
In~\cite[p.~260, Entry~6.4.1]{abram}, the integral representation
\begin{equation}\label{psin}
\psi^{(n)}(z)=(-1)^{n+1}\int_{0}^{\infty}\frac{t^n}{1-\te^{-t}}\te^{-zt}\td t
\end{equation}
is listed for $\Re(z)>0$ and $n\in\mathbb{N}$.
Take
$$
f(x)=\bigl|\psi^{(n)}(x)\bigr|=(-1)^{n+1}\psi^{(n)}(x)>0
$$ 
for $x\in(0,\infty)$ in Lemma~\ref{lem:general}. By~\eqref{psin},
\begin{equation*}
f(x)=\int_{0}^{\infty}\frac{t^n}{1-\te^{-t}}\te^{-xt}\td t>0, \quad x\in(0,\infty).
\end{equation*}
Therefore,
\begin{equation*}
f'(x)=-\int_{0}^{\infty}\frac{t^{n+1}}{1-\te^{-t}}\te^{-xt}\td t<0, \quad x\in(0,\infty)
\end{equation*}
and
\begin{equation*}
H(x)=\frac{\int_{0}^{\infty}\frac{t^n}{1-\te^{-t}}\te^{-xt}\td t} {\int_{0}^{\infty}\frac{t^{n+1}}{1-\te^{-t}}\te^{-xt}\td t}
=-\frac{\psi^{(n)}(x)}{\psi^{(n+1)}(x)}, \quad x\in(0,\infty).
\end{equation*}
\par
Differentiating yields
\begin{equation*}
H'(x)=\frac{\psi^{(n)}(x)\psi^{(n+2)}(x)}{[\psi^{(n+1)}(x)]^2}-1
\end{equation*}
and
\begin{equation*}
H''(x)=\frac{\psi^{(n+2)}(x)}{\psi^{(n+1)}(x)}+\frac{\psi^{(n)}(x) \psi^{(n+3)}(x)}{[\psi^{(n+1)}(x)]^2} -\frac{2\psi^{(n)}(x)[\psi^{(n+2)}(x)]^2}{[\psi^{(n+1)}(x)]^3}.
\end{equation*}
\par
The first statement in~\cite[Theorem~2]{Yang-JMAA-2017} asserts that
the function
\[
\frac{[\psi^{(n+1)}(x)]^2}{\psi^{(n)}(x)\psi^{(n+2)}(x)}
\]
is strictly decreasing on $(0,\infty)$ and maps $(0,\infty)$ onto
\[
\biggl(\frac{n}{n+1},\frac{n+1}{n+2}\biggr),\quad n\in\mathbb{N}.
\]
See also~\cite[Corollary~2.3]{Alzer-Wells-SIAM-98}, \cite[Section~3.5, pp.~10--12]{Qi-Agar-Surv-JIA.tex}, and~\cite[p.~3, Eq.~(8)]{Qi-Nantomah-Lim.tex}.
Consequently, $H'(x)$ is strictly increasing and maps $(0,\infty)$ onto
\[
\biggl(\frac{n+2}{n+1}-1, \frac{n+1}{n}-1\biggr)=\biggl(\frac{1}{n+1},\frac{1}{n}\biggr)\subset(0,1).
\]
In particular, $H''(x)\ge0$ on $(0,\infty)$. Thus, all the hypotheses of Lemma~\ref{lem:general} are satisfied, and the function $\Phi_{n,\alpha}(x)$ defined by~\eqref{Phi-n-alpha(x)} is strictly convex on $(0,\infty)$.
\par
A straightforward computation gives
\begin{align*}
\Phi_{n,\alpha}'(x)&=\frac{\psi^{(n+1)}(x)\psi^{(n+1)}(x+\alpha)}{[\psi^{(n)}(x)-\psi^{(n)}(x+\alpha)]^2} \biggl[\frac{\psi^{(n)}(x)}{\psi^{(n+1)}(x)} -\frac{\psi^{(n)}(x+\alpha)}{\psi^{(n+1)}(x+\alpha)}\biggr]\\
&=\frac{x^n\bigl[x^{n+2}\psi^{(n+1)}(x)\bigr]\psi^{(n+1)}(x+\alpha)}{[x^{n+1}\psi^{(n)}(x)-x^{n+1}\psi^{(n)}(x+\alpha)]^2} \biggl[\frac{x[x^{n+1}\psi^{(n)}(x)]}{x^{n+2}\psi^{(n+1)}(x)} -\frac{\psi^{(n)}(x+\alpha)}{\psi^{(n+1)}(x+\alpha)}\biggr]\\
&\to0
\end{align*}
as $x\to0^+$, where we used the limits $\psi^{(n)}(x+\alpha)\to\psi^{(n)}(\alpha)$ (which is a nonzero finite number), $x^{n+2}\psi^{(n+1)}(x)\to(-1)^{n+2}(n+1)!$, and $x^{n+1}\psi^{(n)}(x)\to(-1)^{n+1}n!$ as $x\to0^+$, valid for $\alpha>0$ and $n\in\mathbb{N}$; see~\cite[Lemma~2.1]{RIMA-D-20-01130.tex}. 
Since $\Phi_{n,\alpha}''(x)>0$, the function $\Phi_{n,\alpha}'(x)$ is strictly increasing on $(0,\infty)$. Moreover, $\lim_{x\to0^+}\Phi_{n,\alpha}'(x)=0$. Hence $\Phi_{n,\alpha}'(x)>0$ for all $x>0$, and therefore $\Phi_{n,\alpha}$ is strictly increasing on $(0,\infty)$.
The proof of Lemma~\ref{OpenProb-convex} is thus complete.
\end{proof}

\begin{rem}
From the proof of Lemma~\ref{OpenProb-convex}, we can conclude that, for $n\in\mathbb{N}$, the inequality
\begin{equation*}
\frac{1}{2}\biggl[\frac{\psi^{(n+2)}(x)}{\psi^{(n+1)}(x)}+\frac{\psi^{(n)}(x) \psi^{(n+3)}(x)}{[\psi^{(n+1)}(x)]^2}\biggr] \ge\frac{\psi^{(n)}(x)[\psi^{(n+2)}(x)]^2}{[\psi^{(n+1)}(x)]^3}
\end{equation*}
is valid on $(0,\infty)$.
\end{rem}

\begin{rem}
Lemma~\ref{OpenProb-convex} provides an affirmative answer to the question posted at~\url{https://mathoverflow.net/q/396837} (accessed on 5 September 2026).
\end{rem}

\begin{rem}
Lemma~4 in~\cite{Yang-Tian-JIA-2017} states that, if the ratio $\frac{A(t)}{B(t)}$ is increasing, then the ratio $\frac{\int_{0}^{\infty}A(t)\te^{-xt}\td t}{\int_{0}^{\infty}B(t)\te^{-xt}\td t}$ is decreasing on $(0,\infty)$, where the functions $A(t)$ and $B(t)\ne0$ are defined on $(0,\infty)$ such that their Laplace transforms exist. A more general ratio monotonicity rule was established in~\cite[Lemma~9 and Remark~15]{Alice-y-x-conj-one.tex} together with~\cite[Remark~7.2]{Ouimet-LCM-BKMS.tex}. By virtue of the integral representation~\eqref{psin}, we have
\begin{equation*}
\frac{\psi^{(n)}(x+\alpha)}{\psi^{(n)}(x)}=\frac{\int_{0}^{\infty}\frac{t^n\te^{-\alpha t}}{1-\te^{-t}}\te^{-xt}\td t} {\int_{0}^{\infty}\frac{t^n}{1-\te^{-t}}\te^{-xt}\td t}.
\end{equation*}
Accordingly, in light of the above-mentioned ratio monotonicity rule~\cite[Lemma~9]{Alice-y-x-conj-one.tex} or~\cite[Lemma~4]{Yang-Tian-JIA-2017}, we conclude that $\frac{\psi^{(n)}(x+\alpha)}{\psi^{(n)}(x)}$ is increasing on $(0,\infty)$. Thus, the monotonic increase of $\Phi_{n,\alpha}(x)$ on $(0,\infty)$ is also established.
\end{rem}

\section{Strict decrease of $\mathcal{K}(x)$}
Now we are in a position to prove our main result.

\begin{thm}\label{sect-curvat-decr-thm}
The sectional curvature $\mathcal{K}(x)$ defined in~\eqref{K(x-x)-Eq} is strictly decreasing on $(0,\infty)$. Moreover, the inequality~\eqref{sect-curvat-ineq} is sharp.
\end{thm}

\begin{proof}
In~\cite[Theorem~3.2]{RIMA-D-20-01130.tex}, the second factor $\frac{\psi''(x)}{[\psi'(x)]^2}$ in~\eqref{K(x-x)-Eq} was proved to be decreasing and to map $(0,\infty)$ onto $(-1,0)$.
\par
The third factor in~\eqref{K(x-x)-Eq} can be reformulated as
\begin{align*}
\frac{2\psi'(x)\psi''(2x)-\psi'(2x)\psi''(x)}{[\psi'(x)-2\psi'(2x)]^2}
&=\frac{\psi'(x)\psi''\bigl(x+\frac{1}{2}\bigr) -\psi'\bigl(x+\frac{1}{2}\bigr)\psi''(x)} {\bigl[\psi'(x)-\psi'\bigl(x+\frac{1}{2}\bigr)\bigr]^2}\\
&=\biggl[\frac{\psi'\bigl(x+\frac12\bigr)}{\psi'(x)}\biggr]' \frac{1}{\bigl[1-\psi'\bigl(x+\frac{1}{2}\bigr)/\psi'(x)\bigr]^2}\\
&=\left[\frac{1}{1-\psi'\bigl(x+\frac{1}{2}\bigr)/\psi'(x)}\right]',
\end{align*}
where we utilized the first two derivatives
\begin{equation*}
\psi'(2z)=\frac{1}{4}\biggl[\psi'(z)+\psi'\biggl(z+\frac{1}{2}\biggr)\biggr]
\end{equation*}
and
\begin{equation*}
\psi''(2z)=\frac{1}{8}\biggl[\psi''(z)+\psi''\biggl(z+\frac{1}{2}\biggr)\biggr]
\end{equation*}
of the duplication formula
\begin{equation*}
\psi(2z)=\frac{1}{2}\biggl[\psi(z)+\psi\biggl(z+\frac{1}{2}\biggr)\biggr]+\ln2
\end{equation*}
in~\cite[p.~259, 6.3.8]{abram}. In view of Lemma~\ref{OpenProb-convex} applied to $n=1$ and $\alpha=\frac{1}{2}$, we conclude that
\begin{equation*}
\biggl[\frac{2\psi'(x)\psi''(2x)-\psi'(2x)\psi''(x)}{[\psi'(x)-2\psi'(2x)]^2}\biggr]'
=\left[\frac{1}{1-\psi'(x+\frac{1}{2})/\psi'(x)}\right]''
>0
\end{equation*}
for $x\in(0,\infty)$. Accordingly, with the aid of the sharp inequality~\eqref{sectional-doub-ineq}, the third factor in~\eqref{K(x-x)-Eq} is positive and strictly increasing on $(0,\infty)$.
\par
In~\eqref{K(x-x)-Eq}, the second factor $\frac{\psi''(x)}{[\psi'(x)]^2}$ is negative and decreasing, whereas the third factor is positive and increasing. Hence $\bigl|\frac{\psi''(x)}{[\psi'(x)]^2}\bigr|$ is positive and increasing. Therefore, the product of the second and third factors is negative and strictly decreasing. Consequently, the sectional curvature $\mathcal{K}(x)$ is strictly decreasing on $(0,\infty)$.
\par
The sharpness of the inequality~\eqref{sect-curvat-ineq} comes from the limits
\begin{equation*}
\lim_{x\to0^+}\mathcal{K}(x)=0 \quad\text{and}\quad \lim_{x\to\infty}\mathcal{K}(x)=-\frac{1}{2}
\end{equation*}
established in the proof of~\cite[Theorem~5.1]{RIMA-D-20-01130.tex}. The proof of Theorem~\ref{sect-curvat-decr-thm} is thus complete.
\end{proof}

\section{Supplementary remarks}
In this section, we discuss several more remarks on our main result.

\begin{rem}
In~\cite[Remark~6.5]{Alice-y-x-curv-notes.tex}, it was conjectured that the function
\begin{equation*}
\frac{2\psi'(x)\psi''(2x)-\psi'(2x)\psi''(x)}{[\psi'(x)-2\psi'(2x)]^2}
\end{equation*}
should be increasing in $x>0$. This suggests that the function
\begin{equation*}
\frac{1}{1-\psi'(x+\frac{1}{2})/\psi'(x)}=\frac{\psi'(x)}{\psi'(x)-\psi'(x+\frac{1}{2})}
\end{equation*}
should be convex in $x\in(0,\infty)$. Lemma~\ref{OpenProb-convex}, together with the proof of Theorem~\ref{sect-curvat-decr-thm}, confirms this conjecture.
\end{rem}

\begin{rem}
From the double inequality
\begin{equation}\label{psi(2k)-doub-ineq}
-\frac{1}{2}\frac{(2k)!}{(k-1)!k!}<\frac{\psi^{(2k)}(x)}{[\psi^{(k)}(x)]^2}<0, \quad x\in(0,\infty), \quad k\in\mathbb{N}
\end{equation}
in~\cite[Theorem~3.2]{RIMA-D-20-01130.tex}, it follows that
\begin{equation}\label{psi(2k=2)-doub-ineq}
-1<\frac{\psi''(x)}{[\psi'(x)]^2}<0, \quad x\in(0,\infty).
\end{equation}
In~\cite[Theorem~11]{Alice-y-x-conj-one.tex}, the double inequalities in~\eqref{psi(2k)-doub-ineq} and~\eqref{psi(2k=2)-doub-ineq} were generalized to a decreasing monotonicity of the function
\begin{equation}\label{Q(m-n)(x)-Q(ijlm)(x)-dfn}
\frac{\psi^{(m+n)}(x)}{\psi^{(m)}(x)\psi^{(n)}(x)}, \quad m,n\ge1
\end{equation}
from $(0,\infty)$ onto $\bigl(-\frac{(m+n-1)!} {(m-1)!(n-1)!},0\bigr)$. See also~\cite[Sections~9 and~12]{manifold-gamma.tex}. In particular, taking $m=n=1$ in~\eqref{Q(m-n)(x)-Q(ijlm)(x)-dfn} shows that the function $\frac{\psi''(x)}{[\psi'(x)]^2}$, which is the second factor in~\eqref{K(x-x)-Eq}, is strictly decreasing and maps $(0,\infty)$ onto $(-1,0)$.
\end{rem}

\begin{rem}
The sectional curvature $K(x,y)$ can be rearranged as
\begin{equation*}
K(x,y)=\frac{1}{4}\frac{\psi''(x)}{[\psi'(x)]^2}\frac{\psi''(y)}{[\psi'(y)]^2}\frac{\psi''(x+y)}{[\psi'(x+y)]^2} \frac{\frac{\psi'(x)}{\psi''(x)}+\frac{\psi'(y)}{\psi''(y)}-\frac{\psi'(x+y)}{\psi''(x+y)}}
{\bigl[\frac{1}{\psi'(x)}+\frac{1}{\psi'(y)}-\frac{1}{\psi'(x+y)}\bigr]^2}.
\end{equation*}
Theorem~3.2 in~\cite{RIMA-D-20-01130.tex} implies that the function $\frac{\psi''(x)}{[\psi'(x)]^2}$ is strictly decreasing and maps $(0,\infty)$ onto $(-1,0)$. Accordingly, 
in order to confirm the strict decrease asserted in Conjecture~\ref{Alice-lower-bound-conj}, show that the function
\begin{align*}
Q(x,y)&=\frac{\psi''(x)}{[\psi'(x)]^2}\frac{\frac{\psi'(x)}{\psi''(x)}+\frac{\psi'(y)}{\psi''(y)}-\frac{\psi'(x+y)}{\psi''(x+y)}}
{\Bigl[\frac{1}{\psi'(x)}+\frac{1}{\psi'(y)}-\frac{1}{\psi'(x+y)}\Bigr]^2}\\
&=\frac{\psi''(x)}{[\psi'(x)]^2}\frac{\frac{\psi'(x+y)}{\psi''(x+y)}}{\Bigl[\frac{1}{\psi'(x+y)}\Bigr]^2} \frac{\frac{\frac{\psi'(x)}{\psi''(x)}+\frac{\psi'(y)}{\psi''(y)}} {\frac{\psi'(x+y)}{\psi''(x+y)}}-1}
{\biggl[\frac{\frac{1}{\psi'(x)}+\frac{1}{\psi'(y)}}{\frac{1}{\psi'(x+y)}}-1\biggr]^2}
\end{align*}
is decreasing in $x\in(0,\infty)$ for fixed $y\in(0,\infty)$.
\par
We note that numerical calculations give $Q(100,1)=-2.0542\dotsc$, and $Q$ is unbounded as $y\to0^+$.
\end{rem}

\begin{rem}
This paper is a revised version of the arXiv preprint at \url{https://doi.org/10.48550/arXiv.2609.10628}. 
\end{rem}

\section{Declarations}

\paragraph{\bf Author Contributions}
The author was solely responsible for the conception, analysis, and preparation of the manuscript.

\paragraph{\bf Funding}
The author was partially supported by the Natural Science Foundation of Inner Mongolia Autonomous Region (Grant No.~2025QN01041) and by the Youth Project of Hulunbuir City for Basic Research and Applied Basic Research (Grant No.~GH2024020).

\paragraph{\bf Institutional Review Board Statement}
Not applicable.

\paragraph{\bf Informed Consent Statement}
Not applicable.

\paragraph{\bf Ethical Approval}
The conducted research is not related to either human or animal use.

\paragraph{\bf Availability of Data and Materials}
Data sharing is not applicable to this article as no new data were created or analyzed in this study.

\paragraph{\bf Acknowledgements}
Not applicable.

\paragraph{\bf Competing Interests}
The author declares no competing interests.

\paragraph{\bf Use of AI tools declaration}
The author declares that no AI tools were used in the creation of this article.

\end{document}